\documentclass[11pt]{article}

\usepackage[margin=1in]{geometry}
\usepackage{amsmath,amssymb,amsthm}
\usepackage{array}
\usepackage{verbatim}
\usepackage[hidelinks]{hyperref}

\theoremstyle{plain}
\newtheorem{theorem}{Theorem}[section]
\newtheorem{lemma}[theorem]{Lemma}
\newtheorem{proposition}[theorem]{Proposition}
\newtheorem{corollary}[theorem]{Corollary}
\newtheorem{conjecture}[theorem]{Conjecture}
\theoremstyle{definition}
\newtheorem{definition}[theorem]{Definition}

\theoremstyle{remark}

\newcommand{\F}{\mathbb F}
\newcommand{\dd}{\delta}
\newcommand{\xor}{\mathbin{\oplus}}
\newcommand{\bigxor}{\mathop{\bigoplus}}
\newcommand{\bits}{\{0,1\}^3}

\title{Three-Bit Flows and Cycle Covers \\ Part I}
\author{
{Shiva Kintali}\footnote{Email : {shiva.kintali@gmail.com}} \\
}
\date{}

\begin{document}
\maketitle

\begin{abstract}
We study the relationship between nowhere-zero three-bit flows, labeled
triangles, and cycle double covers of bridgeless multigraphs.
Using the flow theorems of Seymour and Tutte, we realize
the three flow values at each vertex as the side differences of a triangle
whose sides carry two-element subsets of $\F_2^3$. We express agreement of
these local labels across graph edges as a binary system and prove its
solvability by an inconsistency certificate, a local tester-parity identity,
and a global double count. The resulting compatible labels trace cycle
components in which every edge occurs exactly twice, proving that every finite
bridgeless multigraph has a cycle double cover, {\em proving the cycle
double cover conjecture}.
\end{abstract}

\section{Introduction}
\label{sec:introduction}

The cycle double cover conjecture is a central and long-standing problem in
graph theory: is the absence of bridges the only obstruction to covering every
edge exactly twice by cycles?  The conjecture was formulated independently by
Szekeres and Seymour \cite{Szekeres1973,Seymour1979}; related circuit-cover
questions were studied by Itai and Rodeh \cite{ItaiRodeh1978}. It remained
open in full generality despite extensive work on special classes, flows, and
surface embeddings \cite{Jaeger1985,Zhang2012,GhanbariSamal2026}.
In this paper we prove this conjecture.

All graphs in this paper are finite undirected multigraphs. Parallel edges
and loops are allowed unless explicitly excluded. A loop contributes two to
the degree of its incident vertex. For an edge $e$, the notation $G-e$ means
that $e$ is deleted while its endpoints are retained. An edge $e$ is a
\emph{bridge} if $G-e$ has more connected components than $G$; a graph is
\emph{bridgeless} if it has no bridge. A loop is never a bridge. Isolated
vertices have no effect on cycle double covers.
For a vertex $v$, we write $\dd(v)$ for the set of edges incident with $v$;
because loops are excluded whenever this notation is used below, each such
edge occurs once in $\dd(v)$.

\begin{definition}
A \emph{cycle} is a connected 2-regular subgraph. Thus a loop is a cycle,
and two parallel edges joining the same two vertices form a cycle of length
two. In particular, a cycle uses each of its edges once, although a vertex
of the ambient multigraph may be incident with additional edges not belonging
to the cycle.
\end{definition}

\begin{definition}
A \emph{cycle double cover} of a graph $G$ is a finite multiset $\mathcal C$
of cycles such that every edge of $G$ belongs to exactly two members of
$\mathcal C$, counted with multiplicity. Equal cycles may occur more than
once in $\mathcal C$; their occurrences are regarded as distinct members.
\end{definition}

The empty multiset is a cycle double cover of an edgeless graph. Adding or
deleting isolated vertices changes neither the cycles nor the edge-covering
condition, which justifies ignoring isolated vertices throughout.

A loop is never a bridge. If a non-loop edge lies on a cycle, then after the
edge is deleted the remaining part of the cycle still joins its two endpoints;
such an edge is not a bridge either. In a cycle double cover every edge lies
on a cycle, so no edge can be a bridge. Thus bridgelessness is necessary, and
the cycle double cover conjecture asserts that it is also sufficient.

\begin{conjecture}[Cycle Double Cover Conjecture]
\label{conj:cdc}
Every finite bridgeless multigraph has a cycle double cover.
\end{conjecture}

\subsection{Related work}
Szekeres formulated the problem through polyhedral decompositions of cubic
graphs \cite{Szekeres1973}, while Seymour's circuit formulation
\cite{Seymour1979} and the work of Itai and Rodeh \cite{ItaiRodeh1978} placed it
in the now-standard language of circuit and cycle covers. Jaeger's survey
\cite{Jaeger1985} describes the early theory and the reduction to cubic graphs;
Zhang's monograph \cite{Zhang2012} gives a comprehensive account of circuit
double covers, reductions, and related conjectures.

Several important classes were previously known to satisfy the conjecture.
For a bridgeless plane graph, the facial boundary walks, after decomposition
into cycles, give a cycle double cover \cite{Jaeger1985}. A properly
3-edge-colored cubic graph
also has an immediate cycle double cover: the union of each pair of color
classes is a spanning 2-regular subgraph, and each color class occurs in two
of the three pairs. Alspach, Goddyn, and Zhang established the conjecture for
the broad class of bridgeless graphs containing no subdivision of the Petersen
graph \cite{AlspachGoddynZhang1994}. More recently, the surface-embedding
formulation has motivated quantitative work on embeddings with few singular
edges \cite{GhanbariSamal2026}.

Nowhere-zero flows form another major line of attack. Tutte's group-flow
theorem relates integer flows to flows over finite abelian groups of the same
order \cite{Tutte1954}, and Seymour proved that every bridgeless graph has a
nowhere-zero integer 6-flow \cite{Seymour1981}. The latter is also an integer
8-flow; applying Tutte's theorem with the group $\F_2^3$ therefore gives a
nowhere-zero three-bit flow. These theorems supply the established external
input to our argument. The new step is a global compatibility theorem that
converts the three-bit flow into two-element edge labels satisfying the local
cycle condition. This conversion is the mechanism that completes the cycle
double cover.

Connections between low-order flows and cycle double covers have been studied
in several forms. Zhang related nowhere-zero $4$-flows to cycle double covers
with bounded numbers of cycles and proved further extension results
\cite{Zhang1996}. Xie and Zhang introduced flow-pair covers and used suitable
pairs of integer flows to obtain cycle double covers under additional parity
and support hypotheses \cite{XieZhang2009}. Those results emphasize that a
flow by itself does not automatically display the required cycles. The
flow-to-cover theorem proved here uses a different mechanism: for a cubic
multigraph, an arbitrary nowhere-zero $\F_2^3$-flow is converted into globally
compatible two-element edge labels, from which the cycle components are read
off directly.

As a basic example, suppose that a loopless cubic multigraph has a proper
3-edge-coloring. Every vertex is incident with one edge of each color. The
union of any two color classes is therefore a spanning 2-regular subgraph,
hence a disjoint union of cycles. There are three choices of two color
classes. An edge of one fixed color belongs to exactly the two choices
containing that color, so the cycle components of these three 2-factors form a
cycle double cover. Thus the non-3-edge-colorable cubic multigraphs
constituted the essential unresolved case after the cubic reduction.

\subsection{Proof strategy and roadmap}
Our proof converts a nowhere-zero three-bit flow into a cycle double cover by
passing through a system of locally defined, globally compatible two-element
labels. The argument has six stages.
\begin{enumerate}
\item \emph{Structural reduction.}
Section~\ref{sec:elementary-lemmas} establishes the elementary cycle and
trail facts needed for the reduction. Proposition~\ref{prop:cubic-reduction}
in Section~\ref{sec:cubic-reduction} replaces an arbitrary bridgeless
multigraph by a loopless bridgeless cubic multigraph in such a way that a
cycle double cover of the cubic graph projects to a cycle double cover
of the original graph. It therefore suffices to prove the conjecture in
the cubic case.

\item \emph{Three-bit flow input.}
Section~\ref{sec:flows} combines Seymour's 6-flow theorem
(Theorem~\ref{thm:seymour}) with Tutte's group-flow theorem
(Theorem~\ref{thm:tutte}). Their consequence,
Corollary~\ref{cor:three-bit-flow}, supplies a nowhere-zero
$\F_2^3$-flow $f$. At a cubic vertex its three incident values are
distinct nonzero vectors whose XOR is zero, as recorded in
Lemma~\ref{lem:fano-line}.

\item \emph{Local triangle labels.}
In Section~\ref{sec:vertex-triangle}, the three flow values at each
vertex are realized as the XOR differences of the sides of a triangle.
A translation $t_v\in\F_2^3$ moves all three triangle labels
simultaneously. Lemma~\ref{lem:triangle-parity} proves the crucial
local invariant: every label occurs on either zero or two sides of the
triangle. Lemma~\ref{lem:uniform-side-formula} expresses every side in
a form suitable for comparison across a graph edge.

\item \emph{The global gluing system.}
Section~\ref{sec:edge-gluing} asks for translations $t_v$ that make the
two endpoint triangles assign the same unordered label pair to each
edge. Proposition~\ref{prop:edge-gluing-equation} shows that this is
equivalent to the binary system
\[
t_u\xor t_v\xor\epsilon_e f(e)=d_e
\qquad(e=uv).
\]
Thus the geometric compatibility question becomes a finite system of
linear equations over $\F_2$.

\item \emph{Elimination of compatibility obstructions.}
Section~\ref{sec:xor-certificates} uses
Lemma~\ref{lem:xor-certificate} to represent any inconsistency by an XOR
certificate $0=1$. Lemma~\ref{lem:certificate-shape} encodes such a
certificate by one tester vector on each edge. The local calculation in
Section~\ref{sec:local-testers}, especially
Lemma~\ref{lem:local-tester-parity}, identifies the discrepancy at a
vertex with the parity of its incident nonzero testers. In
Section~\ref{sec:triangle-compatibility}, every nonzero edge tester is
counted once at each endpoint. The resulting incidence count is even,
so the global discrepancy is zero in $\F_2$. This contradicts the
required value one in an inconsistency certificate and proves the
triangle compatibility theorem,
Theorem~\ref{thm:triangle-compatibility}.

\item \emph{Extraction of the cycle double cover.}
Section~\ref{sec:cycles} assigns to each label $s\in\F_2^3$ the subgraph
consisting of the edges whose common two-element label set contains $s$.
Lemma~\ref{lem:labels-trace-cycles} uses triangle parity to show that
every vertex has degree zero or two in this subgraph, so every
edge-containing component is a cycle. The common two-element label set
on each edge contains exactly two distinct labels and hence places that
edge in exactly two indexed cycle components.
Corollary~\ref{thm:cubic-conclusion} proves the cubic case, and
Proposition~\ref{prop:cubic-reduction} then proves
Conjecture~\ref{conj:cdc} in full.
\end{enumerate}

Theorems~\ref{thm:seymour} and~\ref{thm:tutte} are the two external inputs.
The principal new result is Theorem~\ref{thm:triangle-compatibility}; all
reductions and all subsequent steps converting the resulting three-bit flow
into a cycle double cover are proved here. The resulting flow-to-cover
statement is isolated as Theorem~\ref{thm:flow-to-cdc}.

\section{Preliminaries}
\label{sec:elementary-lemmas}

A \emph{trail} is a sequence $v_0e_1v_1e_2\cdots e_mv_m$
in which each edge $e_i$ joins $v_{i-1}$ to $v_i$ and the edge objects
$e_1,\ldots,e_m$ are pairwise distinct. It is \emph{closed} if $v_m=v_0$.
Parallel edges are distinct edge objects. To say that a trail decomposes into
edge-disjoint cycles means that the cycles have pairwise disjoint edge sets and
that the union of those edge sets is exactly the set of edges in the trail.

\begin{lemma}[Non-bridges lie on cycles]
\label{lem:edge-cycle}
Let $G$ be a loopless multigraph and let $e=uv$ be an edge. Then $e$ is not
a bridge if and only if $e$ lies on a cycle.
\end{lemma}

\begin{proof}
Suppose first that $e$ lies on a cycle $D$. Since $G$ is loopless, $u$ and
$v$ are distinct. Deleting $e$ from $D$ leaves a $u$--$v$ path in $G-e$.
Let $x$ and $y$ be vertices in the same component of $G$, and choose an
$x$--$y$ path $Q$ in $G$. If $Q$ avoids $e$, it remains in $G-e$. If $Q$
uses $e$, replace that occurrence of $e$ by the $u$--$v$ path $D-e$. The
result is an $x$--$y$ walk in $G-e$, and every walk contains a path between its
endpoints. Thus deleting $e$ does not separate any two previously connected
vertices and cannot increase the number of components. Hence $e$ is not a
bridge.

Conversely, suppose $e$ is not a bridge. If $u$ and $v$ belonged to different
components of $G-e$, then deleting $e$ would split the component of $G$
containing $e$ into at least two components, so $e$ would be a bridge.
Therefore $u$ and $v$ lie in the same component of $G-e$. Choose a
$u$--$v$ path $P$ in $G-e$. The union of $P$ and $e$ is connected, and every
vertex in that union has degree two. It is therefore a cycle containing $e$.
This argument also covers parallel edges: if $P$ consists of one edge
parallel to $e$, then $P\cup\{e\}$ is a cycle of length two.
\end{proof}

\begin{lemma}[A closed trail splits into cycles]
\label{lem:trail-cycles}
Every nonempty closed trail in a loopless multigraph decomposes into
edge-disjoint cycles; equivalently, its edge set is the disjoint union of the
edge sets of those cycles.
\end{lemma}

\begin{proof}
We use strong induction on the number $m$ of edges in the trail. Write the
closed trail as
\[
W=v_0e_1v_1e_2\cdots e_mv_m,
\qquad v_m=v_0,
\]
where $m\geq1$ and the edges $e_1,\ldots,e_m$ are pairwise distinct, and
assume the assertion holds for every nonempty closed trail with fewer than
$m$ edges. If no vertex is repeated except $v_0=v_m$, then the subgraph
traced by $W$ is connected and 2-regular. Hence that subgraph is a cycle, and
its edge set is the edge set of $W$.

Otherwise choose indices $0\leq i<j\leq m$ with $v_i=v_j$ and with $j-i$
as small as possible among all repeated-vertex pairs. The segment
\[
C=v_ie_{i+1}v_{i+1}\cdots e_jv_j
\tag{2.1}
\]
is a closed trail. It has no repeated internal vertex: any such repetition,
or any internal occurrence of $v_i$, would give a repeated-vertex pair with
smaller index difference. Moreover, $j-i\geq2$, since $j-i=1$ would make
$e_j$ a loop. Every internal vertex of the segment is incident in the traced
subgraph with its entering and leaving edges, while $v_i=v_j$ is incident with
the distinct first and last edges. Thus the traced subgraph is connected and
2-regular, so it is a cycle. We denote this cycle subgraph also by $C$.

Delete the segment (2.1) from the displayed order for $W$ and concatenate the
remaining portions at the common vertex $v_i=v_j$. The resulting sequence is
\[
W'=v_0e_1\cdots e_iv_i(=v_j)e_{j+1}\cdots e_mv_m.
\tag{2.2}
\]
The initial portion is omitted when $i=0$, and the final portion is omitted
when $j=m$. If no edge remains, then $C$ uses every edge of $W$ and we are
done. Otherwise (2.2) starts and ends at $v_0=v_m$ and uses pairwise distinct
edges, so it is a closed trail. It has $m-(j-i)<m$ edges. By the induction
hypothesis, $W'$ decomposes into edge-disjoint cycles. Those cycles are also
edge-disjoint from $C$, because $W'$ uses exactly the edges of $W$ not used by
$C$. Together with $C$, they give the required decomposition of $W$.
\end{proof}

\section{Reduction to cubic graphs}
\label{sec:cubic-reduction}

This is a standard reduction. We present the details here for the sake of completeness.

\begin{proposition}[Cubic reduction]
\label{prop:cubic-reduction}
Suppose every loopless bridgeless cubic multigraph has a cycle double cover.
Then every finite bridgeless multigraph has a cycle double cover.
\end{proposition}

\begin{proof}
The construction may be carried out simultaneously in all connected
components. Equivalently, one may apply it componentwise and take the
multiset union of the resulting covers. We proceed in four steps.

\medskip
\noindent\textbf{Step 1: remove loops.}
Let $L$ be the set of loop edges of $G$, and let $G_0=G-L$. A loop does not
join two distinct vertices and cannot occur on a path between distinct
vertices. Consequently, for every non-loop edge $e$ and every pair of
vertices $x,y$, the vertices $x$ and $y$ are connected in $G-e$ if and only if
they are connected in $G_0-e$. Thus $G-e$ and $G_0-e$ have the same connected
components, and $e$ is a bridge of $G$ if and only if it is a bridge of $G_0$.
Since $G$ is bridgeless, $G_0$ is loopless and bridgeless.

If $G_0$ has no edges, begin with the empty cycle double cover of $G_0$ and
skip Steps 2--4, proceeding directly to restoration of the loops at the end of
the proof. Hence assume that $G_0$ has an edge. Isolated vertices of $G_0$
may be ignored. Every remaining vertex has degree at least two: if $e$ were
incident with a vertex of degree one, deleting $e$ would isolate that vertex
from the other endpoint of $e$, making $e$ a bridge.

\medskip
\noindent\textbf{Step 2: expand every vertex into a cycle.}
For each nonisolated vertex $v$ of $G_0$, let $d=d_{G_0}(v)\geq2$ and replace
$v$ by a new cycle, taking these replacement cycles to be pairwise
vertex-disjoint:
\[
R_v=v_1v_2\cdots v_dv_1.
\]
When $d=2$, this means two parallel edges joining $v_1$ and $v_2$. Attach
the $d$ edge-ends incident with $v$ bijectively to the $d$ vertices of
$R_v$. More formally, choose a bijection from the incidences $(v,e)$ with
$e\in\dd_{G_0}(v)$ to $V(R_v)$. If $e=uv$ is an edge of $G_0$, join the
chosen attachment vertex in $R_u$ to the chosen attachment vertex in $R_v$.
Denote this new edge by $\widehat e$.

Call the edges $\widehat e$ \emph{old edges} and the edges on the cycles
$R_v$ \emph{replacement edges}. Let $H$ be the resulting graph. Each
vertex of $H$ is incident with two replacement edges and exactly one old
edge, so $H$ is cubic. It is loopless: replacement edges lie on loopless
cycles $R_v$, including the length-two case, and an old edge $\widehat e$
joins $R_u$ to $R_v$ for distinct original endpoints $u\neq v$ because
$G_0$ is loopless.

\medskip
\noindent\textbf{Step 3: prove that $H$ is bridgeless.}
Every replacement edge lies on the cycle $R_v$ containing it. Now consider
an old edge $\widehat e$ corresponding to $e\in E(G_0)$. Since $G_0$ is
loopless and bridgeless, Lemma~\ref{lem:edge-cycle} gives a cycle $D$ of
$G_0$ containing $e$.

We lift $D$ to $H$. Include $\widehat a$ for every edge $a\in E(D)$. At a
vertex $v$ of $D$, exactly two distinct edges of $D$ are incident with $v$.
Their edge-ends were attached to two distinct vertices of $R_v$, because the
attachment map is a bijection. The cycle $R_v$ is the union of two
edge-disjoint paths between those attachment vertices. Choose either path and
include all its replacement edges. We call the chosen path an arc of $R_v$.
In the length-two case, the two arcs are the two individual parallel
replacement edges, so the same description applies.

The resulting subgraph is connected: it is obtained from the connected cycle
$D$ by replacing each vertex with a joining path. At an internal vertex of a
chosen arc, exactly two chosen replacement edges are incident. At an endpoint
of a chosen arc, exactly one chosen replacement edge and one chosen old edge
are incident. No other vertices are used. Hence every vertex of the lifted
subgraph has degree two. The lift is therefore a cycle of $H$, and it
contains $\widehat e$. Thus every edge of $H$ lies on a cycle. By
Lemma~\ref{lem:edge-cycle}, $H$ is bridgeless.

\medskip
\noindent\textbf{Step 4: contract a cover of $H$ back to $G$.}
By the hypothesis of the proposition, $H$ has a cycle double cover
$\mathcal D$. This hypothesis applies whether or not $H$ is connected.
Define a projection from $H$ to $G_0$ by contracting every replacement cycle
$R_v$ to the original vertex $v$, sending each old edge $\widehat e$ to its
original edge $e$, and suppressing all replacement edges.

Consider one occurrence of a cycle $Z$ in the multiset $\mathcal D$. If $Z$
uses no old edge, then its connectedness forces it to lie entirely in one
replacement cycle; its projection is a single vertex, so discard it. Suppose
instead that $Z$ uses an old edge $\widehat e$ with endpoints in distinct
replacement cycles. If this were the only old edge of $Z$, then deleting
$\widehat e$ from $Z$ would leave a path between its endpoints using only
replacement edges. This is impossible because distinct replacement cycles
are distinct components of the replacement-edge subgraph. Hence $Z$ uses at
least two old edges. Read them in their cyclic order. Each vertex of $H$ is
incident with only one old edge, so between two successive old edges, $Z$
follows a nonempty path of replacement edges within a single cycle $R_v$.
After contraction, the two old edges are therefore incident with the same
original vertex $v$. The cyclic sequence of old edges consequently becomes a
closed walk in $G_0$. It may visit an original vertex more than once if $Z$
passes through the corresponding replacement cycle in more than one segment.

No old edge occurs twice in this walk, because $Z$ is a cycle and hence uses
each edge of $H$ at most once. The closed walk is therefore a closed trail.
By Lemma~\ref{lem:trail-cycles}, decompose it into edge-disjoint cycles of
$G_0$, and insert all those cycles into a new multiset $\mathcal C_0$. Carry
out this operation separately for every occurrence of every member of
$\mathcal D$; if equal cycles result from different occurrences, retain their
multiplicity.

Each edge $e\in E(G_0)$ has exactly one old representative $\widehat e$ in
$H$. The edge $\widehat e$ occurs in exactly two members of $\mathcal D$,
counted with multiplicity. Each occurrence contributes one occurrence of
$e$ to exactly one cycle produced by the corresponding trail decomposition.
No other operation creates or deletes an old-edge occurrence. Hence every
edge of $G_0$ occurs in exactly two members of $\mathcal C_0$, so
$\mathcal C_0$ is a cycle double cover of $G_0$. The isolated vertices ignored
in Step 2 have no incident edges and impose no additional covering condition.

Finally, insert two copies of each loop in $L$. Each loop is itself a cycle,
and no cycle of $G_0$ contains a deleted loop. Thus every restored loop and
every non-loop edge of the original graph is covered exactly twice. The
result is a cycle double cover of $G$.
\end{proof}

From now on, $G$ denotes a loopless bridgeless cubic multigraph.

\section{Three-bit colors and flows}
\label{sec:flows}

Let $\F_2=\{0,1\}$ be the field with arithmetic modulo $2$. The set
$\bits=\{000,001,010,011,100,101,110,111\}$
is the three-dimensional vector space $\F_2^3$. We write its addition as
$\xor$ and call it \emph{XOR}. The zero vector is $000$, and addition is
performed coordinate by coordinate modulo $2$. For example,
$100\xor010=110,\ 101\xor101=000$.
Thus every $s\in\bits$ satisfies $s\xor s=000$, or equivalently $-s=s$.

\begin{definition}
Let $A$ be an abelian group, written additively, and choose an orientation of
each edge of a graph. Interpret an empty sum as the identity element
$0\in A$. An \emph{$A$-flow} is a map $f\colon E(G)\to A$ such that, at every
vertex $v$,
\[
\sum_{\substack{e\in E(G)\\ e\text{ directed out of }v}}f(e) = 
\sum_{\substack{e\in E(G)\\ e\text{ directed into }v}}f(e)
\tag{4.1}.
\]
It is \emph{nowhere-zero} if $f(e)\neq0$ for every edge $e$. A
\emph{three-bit flow} is an $A$-flow with $A=\F_2^3$; in that case the sums in
\textup{(4.1)} are XOR sums. If loops are present, an oriented loop has its
tail and head at the same vertex and therefore contributes once to each side
of \textup{(4.1)}. In the main argument, however, $G$ is loopless by the
standing assumption following Proposition~\ref{prop:cubic-reduction}.
\end{definition}

\begin{definition}
An \emph{integer flow} is an orientation of the edges together with a map
$\varphi\colon E(G)\to\mathbb Z$ satisfying the conservation equation
\textup{(4.1)} with ordinary integer addition. For an integer $k\geq2$, it is
an \emph{integer $k$-flow} if $|\varphi(e)|<k$ for every edge $e$, and it is
\emph{nowhere-zero} if $\varphi(e)\neq0$ for every edge $e$. Thus the values
of a nowhere-zero integer $k$-flow belong to
$\{\pm1,\ldots,\pm(k-1)\}$.
\end{definition}

For a flow with values in any abelian group, reversing one oriented edge and
replacing its value by its additive inverse preserves every conservation
equation. Consequently, existence of a nowhere-zero flow is independent of
the chosen orientation.

For $A=\F_2^3$, every value is its own additive inverse, so an edge value does
not even have to be changed when its orientation is reversed. More directly,
XORing the incoming side of \textup{(4.1)} onto both sides gives
\[
\left(\bigxor_{e\text{ directed out of }v}f(e)\right)
\xor
\left(\bigxor_{e\text{ directed into }v}f(e)\right)=000.
\]
Because $G$ is loopless, the edges incident with $v$ are partitioned by these two
indexing sets. Hence \textup{(4.1)} is equivalent to
\[
\bigxor_{e\in\dd(v)} f(e)=000.
\tag{4.2}
\]
At a cubic vertex, equation \textup{(4.2)} is the XOR-zero rule for its three
incident edge values.

We use two standard theorems as black boxes.

\begin{theorem}[Seymour's 6-flow theorem \cite{Seymour1981}]
\label{thm:seymour}
Every bridgeless graph has a nowhere-zero integer 6-flow.
\end{theorem}

\begin{theorem}[Tutte's group-flow theorem
\cite{Tutte1954,Zhang2012}]
\label{thm:tutte}
Let $k\geq2$, and let $A$ be any finite abelian group of order $k$. A finite
graph has a nowhere-zero $A$-flow if and only if it has a nowhere-zero integer
$k$-flow.
\end{theorem}

A nowhere-zero integer 6-flow is also a nowhere-zero integer 8-flow: the same
orientation and the same edge values satisfy
$0<|\varphi(e)|<6<8$ for every edge $e$. Since $\F_2^3$ is an abelian group
of order eight, the two theorems give the following input.

\begin{corollary}[Three-bit flow input]
\label{cor:three-bit-flow}
Every bridgeless graph has a nowhere-zero three-bit flow.
\end{corollary}

\begin{proof}
Use Theorem~\ref{thm:seymour} to obtain a nowhere-zero integer 6-flow, regard
it as an integer 8-flow, and apply Theorem~\ref{thm:tutte} with $k=8$ and
$A=\F_2^3$.
\end{proof}

\begin{lemma}[The local Fano-line rule]
\label{lem:fano-line}
Let $f$ be a nowhere-zero three-bit flow on a loopless graph $G$, let $v$ be a
cubic vertex, and let $x,y,z$ be the values of the three edges incident with
$v$, listed in any order. Then $x,y,z$ are nonzero and pairwise distinct, and
$z=x\xor y$ for this ordering.
\end{lemma}

\begin{proof}
Equation \textup{(4.2)} gives $x\xor y\xor z=000$. Therefore
$x\xor y=(x\xor y\xor z)\xor z=000\xor z=z$.
The three values are nonzero because $f$ is nowhere-zero. If $x=y$, then
$z=x\xor y=000$, a
contradiction. If $x=z$, then the flow equation gives $y=000$, and if $y=z$,
it gives $x=000$. Hence the values are pairwise distinct.
\end{proof}

For completeness, the points of the projective plane over $\F_2$ may be
identified with the seven nonzero vectors in $\F_2^3$: each one-dimensional
subspace has a unique nonzero vector. Because distinct nonzero vectors $x$
and $y$ are linearly independent, their span is
$\{000,x,y,x\xor y\}$. Its three nonzero vectors are therefore
$x,y,x\xor y=z$, and the corresponding three points form a line in the Fano
plane. This explains the name of the lemma; no projective geometry is used
later.

\section{The triangle at a cubic vertex}
\label{sec:vertex-triangle}

Fix a nowhere-zero three-bit flow $f$. Because $G$ is loopless and cubic, the
three members of $\dd(v)$ are three distinct edge objects, even when some of
them are parallel in $G$. At every vertex $v$, choose an arbitrary ordering
$a_v,b_v,c_v$ of these edges. When one vertex is fixed, abbreviate this
ordering to $a,b,c$ and put
\[
x=f(a),\qquad y=f(b),\qquad z=f(c)=x\xor y.
\]
By Lemma~\ref{lem:fano-line}, $x,y,z$ are nonzero and pairwise distinct. The
identity $z=x\xor y$ holds for every ordering $a,b,c$, so the construction
does not require a specially chosen order.

For $t\in\bits$, the map $\tau_t\colon\bits\to\bits$ defined by
$\tau_t(s)=t\xor s$ is a bijection, since
\[
\tau_t(\tau_t(s))=t\xor(t\xor s)=(t\xor t)\xor s=s.
\]
Thus $\tau_t$ is its own inverse. We call it translation by $t$. Apply this
translation to the three labels $000,x,z$ to obtain
\[
t,\qquad t\xor x,\qquad t\xor z.
\tag{5.1}
\]
The labels $000,x,z$ are distinct by Lemma~\ref{lem:fano-line}, and a bijection
preserves distinctness, so the three labels in \textup{(5.1)} are distinct.
Regard them as the vertices of an abstract triangle. For an unordered pair of
labels, its \emph{XOR difference} means the XOR of its two elements. XOR
difference is invariant under simultaneous translation: for all
$p,q,t\in\bits$, $(t\xor p)\xor(t\xor q)=p\xor q$.
Assign the three sides to $a,b,c$ as follows:
\[
\begin{array}{c|c|c}
\text{graph edge}&\text{two labels on its triangle side}&
\text{XOR difference}\\
\hline
a&\{t,t\xor x\}&x=f(a)\\
b&\{t\xor x,t\xor z\}&x\xor z=y=f(b)\\
c&\{t,t\xor z\}&z=f(c)
\end{array}
\tag{5.2}
\]
The first and third rows follow immediately from the definitions of $x$ and
$z$. For the middle row,
\[
(t\xor x)\xor(t\xor z)=x\xor z
=x\xor(x\xor y)=y.
\tag{5.3}
\]
Thus the XOR difference of the side assigned to each edge $e\in\dd(v)$ is
exactly its flow value $f(e)$. In particular, every side in \textup{(5.2)}
contains two distinct labels because its XOR difference is nonzero.

\begin{lemma}[Triangle parity]
\label{lem:triangle-parity}
For every $s\in\bits$, exactly zero or two of the three side sets in
\textup{(5.2)} contain $s$.
\end{lemma}

\begin{proof}
Write $\alpha=t,\ \beta=t\xor x,\ \gamma=t\xor z$.
The three labels are distinct, and the three side sets in \textup{(5.2)} are
$\{\alpha,\beta\}$, $\{\beta,\gamma\}$, and
$\{\gamma,\alpha\}$. Each of $\alpha,\beta,\gamma$ therefore belongs to
exactly two side sets. Every $s\in\bits\setminus
\{\alpha,\beta,\gamma\}$ belongs to none of them. These are all eight
possible cases because $\bits$ is the disjoint union of
$\{\alpha,\beta,\gamma\}$ and its complement. This proves the assertion.
\end{proof}

For every incidence $(v,e)$, meaning that $e\in\dd(v)$, define the
\emph{local offset} $g_{v,e}\in\bits$ by
\[
g_{v,e}=
\begin{cases}
000,&e=a_v,\\
f(a_v),&e=b_v,\\
000,&e=c_v.
\end{cases}
\tag{5.4}
\]
The three cases are mutually exclusive and exhaustive because
$a_v,b_v,c_v$ are distinct and together form $\dd(v)$.
These offsets depend only on the fixed flow and the chosen local edge
ordering; they do not depend on the translation parameter $t$.

\begin{lemma}[Uniform side formula]
\label{lem:uniform-side-formula}
For every vertex $v$, every incident edge $e\in\dd(v)$, and every
$t\in\bits$, the side assigned to $e$ in \textup{(5.2)} is
\[
Q_{v,e}(t)
=\{t\xor g_{v,e},\ t\xor g_{v,e}\xor f(e)\}.
\tag{5.5}
\]
This is a two-element subset of $\bits$ whose XOR difference is $f(e)$.
\end{lemma}

\begin{proof}
Fix $v$ and again write $a,b,c$ and $x,y,z$ for the local ordering and flow
values. If $e=a$, then $g_{v,a}=000$, and \textup{(5.5)} is
$\{t,t\xor x\}$. If $e=b$, then $g_{v,b}=x$, and
\[
Q_{v,b}(t)=\{t\xor x,t\xor x\xor f(b)\}
=\{t\xor x,t\xor x\xor y\}
=\{t\xor x,t\xor z\}.
\]
If $e=c$, then $g_{v,c}=000$, and \textup{(5.5)} is
$\{t,t\xor z\}$. These are precisely the three rows of \textup{(5.2)}.
Finally, the XOR difference of the two entries in \textup{(5.5)} is
\[
(t\xor g_{v,e})\xor
(t\xor g_{v,e}\xor f(e))=f(e),
\]
because the two copies of both $t$ and $g_{v,e}$ cancel. This value is
nonzero because $f$ is nowhere-zero, so the two entries are distinct.
\end{proof}

The offset belongs to the incidence $(v,e)$, not to the edge $e$ alone.
Consequently, the two endpoints of one edge may assign it different offsets;
the next section addresses exactly this compatibility issue.

\section{Edge-gluing}
\label{sec:edge-gluing}

Continue with the loopless cubic multigraph $G$, the nowhere-zero three-bit
flow $f$, and the local edge orderings fixed in
Section~\ref{sec:vertex-triangle}. Each vertex has its own labeled triangle.
We seek one translation $t_v\in\bits$ for every vertex $v$ such that, for
every edge $e=uv$, the side assigned to $e$ by the triangle at $u$ is the
same unordered two-element set as the side assigned to $e$ by the triangle
at $v$.

For a bit $\epsilon\in\{0,1\}=\F_2$ and a vector $s\in\bits$, the notation
$\epsilon s$ means scalar multiplication in the vector space $\F_2^3$.
Equivalently, multiplication is performed in each coordinate; thus
$0s=000$ and $1s=s$.

\begin{lemma}[When two labeled sides agree]
\label{lem:pair-agreement}
Let $A,B,p\in\bits$ with $p\neq000$. Then
$\{A,A\xor p\}=\{B,B\xor p\}$ if and only if $A\xor B\in\{000,p\}$.
When these conditions hold, there is a unique bit $\epsilon\in\{0,1\}$ such
that $A\xor B=\epsilon p$.
\end{lemma}

\begin{proof}
Because $p\neq000$, both displayed sets have two elements. If the unordered
pairs agree, then $A$ belongs to $\{B,B\xor p\}$, so $A=B$ or
$A=B\xor p$. Therefore $A\xor B$ is $000$ or $p$.

Conversely, $A\xor B=000$ implies $A=B$, and the pairs agree term by term.
If $A\xor B=p$, then $A=B\xor p$ and $A\xor p=B$, so the same two elements
occur in the opposite order. Finally, $000\neq p$, so exactly one of the two
possibilities holds. Taking $\epsilon=0$ in the first case and $\epsilon=1$
in the second proves existence and uniqueness of $\epsilon$.
\end{proof}

Fix an edge $e=uv$. The standing loopless assumption gives $u\neq v$.
Define its \emph{offset discrepancy}
\[
d_e=g_{u,e}\xor g_{v,e}.
\tag{6.1}
\]
This definition is independent of whether the endpoints are written as
$u,v$ or as $v,u$, because XOR is commutative.

\begin{proposition}[The edge-gluing equation]
\label{prop:edge-gluing-equation}
Fix translations $t_v\in\bits$ at all vertices, and fix an edge $e=uv$.
The two endpoint triangles assign the same side set to $e$ if and only if
there exists a bit $\epsilon_e\in\{0,1\}$ such that
\[
t_u\xor t_v\xor \epsilon_e f(e)=d_e.
\tag{6.2}
\]
If such a bit exists, it is unique.
\end{proposition}

\begin{proof}
By Lemma~\ref{lem:uniform-side-formula}, the endpoint side sets are
\[
\{A_e,A_e\xor f(e)\}
\quad\hbox{and}\quad
\{B_e,B_e\xor f(e)\},
\]
where
\[
A_e=t_u\xor g_{u,e},
\qquad
B_e=t_v\xor g_{v,e}.
\]
The value $f(e)$ is nonzero. Lemma~\ref{lem:pair-agreement}, applied with
$p=f(e)$, therefore says that the two side sets agree if and only if there is
a unique $\epsilon_e\in\{0,1\}$ for which
\[
A_e\xor B_e=\epsilon_e f(e).
\]
Using \textup{(6.1)}, the left side is
\[
A_e\xor B_e
=t_u\xor t_v\xor g_{u,e}\xor g_{v,e}
=t_u\xor t_v\xor d_e.
\]
Consequently, side agreement is equivalent to the following chain of
equivalent equations:
\[
\begin{aligned}
t_u\xor t_v\xor d_e&=\epsilon_e f(e),\\
t_u\xor t_v\xor d_e\xor\epsilon_e f(e)&=000,\\
t_u\xor t_v\xor\epsilon_e f(e)&=d_e.
\end{aligned}
\]
The last equation is \textup{(6.2)}. Each passage merely XORs the same
vector with both sides and is therefore reversible. The uniqueness of
$\epsilon_e$ is the uniqueness supplied by
Lemma~\ref{lem:pair-agreement}.
\end{proof}

Equation \textup{(6.2)} is unchanged if $u$ and $v$ are interchanged, so it
requires no orientation of $e$.
We call the family of equations \textup{(6.2)}, one for every edge, the
\emph{triangle-gluing system}. Its unknowns are the three coordinate bits of
each $t_v$ and one bit $\epsilon_e$ for each edge; the values $f(e)$ and $d_e$
are fixed data. If $f(e)=(f_1(e),f_2(e),f_3(e))$ and similarly for $t_v$ and
$d_e$, then the $i$th coordinate of \textup{(6.2)} is
\[
(t_u)_i\xor(t_v)_i\xor f_i(e)\epsilon_e=(d_e)_i
\qquad(i=1,2,3).
\]
Thus every vector equation is exactly three scalar linear equations over
$\F_2$. Applying Proposition~\ref{prop:edge-gluing-equation} separately to
each edge shows that a family $(t_v)_{v\in V(G)}$ makes all endpoint side sets
agree simultaneously if and only if it can be extended by bits
$(\epsilon_e)_{e\in E(G)}$ to a solution of the triangle-gluing system.
Parallel edges are distinct members of $E(G)$ and therefore contribute
distinct equations, each with its own discrepancy $d_e$ and unknown
$\epsilon_e$.
Hence the vertex triangles can be glued along every edge simultaneously if
and only if this binary system has a solution.

The rest of the main proof shows that this system is always consistent.

\section{A certificate for inconsistent XOR equations}
\label{sec:xor-certificates}

We now analyze the solvability of the triangle-gluing system from
Section~\ref{sec:edge-gluing}. We begin with a basic fact about finite binary
linear systems.

\begin{lemma}[XOR inconsistency certificate]
\label{lem:xor-certificate}
Consider a finite system of linear equations over $\F_2$, written using XOR.
Thus every equation has the form
\[
\bigxor_{j\in J}x_j=b,
\]
where the $x_j$ are bit-valued unknowns, $J$ is a set of indices, and
$b\in\{0,1\}$; variables with coefficient zero are omitted, and the XOR over
the empty set is $0$. The system is inconsistent if and only if the XOR of a
subfamily of its equations has every unknown cancel from the left side and
has right side equal to $1$. Equivalently, an inconsistent system has a
certificate of the form $0=1$.
\end{lemma}

\begin{proof}
If XORing a subfamily of equations gives $0=1$, no assignment can satisfy all
of those equations, so the system is inconsistent.

Conversely, write the system as an augmented matrix $[M\mid b]$ over $\F_2$
and perform Gaussian elimination to row-echelon form. Row interchanges
merely reorder the equations. Replacing a row $R_i$ by $R_i\xor R_j$ with
$i\neq j$ also preserves the solution set: an assignment satisfies $R_i$ and
$R_j$ if and only if it satisfies $R_i\xor R_j$ and $R_j$. These are the only
nontrivial elementary row operations over $\F_2$.

Initially, every row is the XOR of a one-element subfamily of the original
rows. This property is preserved by row interchange, and replacing two such
XORs by their XOR corresponds to taking the symmetric difference of their
index sets. Hence every row produced by elimination is the XOR of a
subfamily of the original rows.

If row-echelon form contains $[0\ \cdots\ 0\mid1]$, then it contains the false
equation $0=1$ and is inconsistent. Conversely, if no such row occurs, set
all free variables to zero and solve successively for the pivot variables,
starting with the last nonzero row. This gives a solution, so the system is
consistent. Therefore an inconsistent system produces a row
$[0\ \cdots\ 0\mid1]$, and the corresponding XOR of original equations is the
required certificate.
\end{proof}

For example, if $p$ and $q$ are bit-valued unknowns, the equations
$p\xor q=0, \ p\xor q=1$ are inconsistent. XORing them cancels both unknowns and gives $0=1$. The
certificate lemma says that every inconsistent binary system has a possibly
larger certificate of exactly this kind.

Here is the matrix form of the argument that follows. Order the scalar
unknowns by collecting the three coordinates of every $t_v$ and then the bits
$\epsilon_e$, and write
\[
\xi\in\F_2^{3|V(G)|+|E(G)|}
\]
for the resulting column vector. Order the three coordinate equations of
\textup{(6.2)} for all edges, concatenate the coordinates of the discrepancies
$d_e$ into a column vector
\[
\mathbf d\in\F_2^{3|E(G)|},
\]
and let $A$ be the corresponding coefficient matrix. The triangle-gluing
system is then
\[
A\xi=\mathbf d.
\tag{7.1}
\]
Lemma~\ref{lem:xor-certificate} says equivalently that this system is
inconsistent exactly when there is a row-selection vector
$y\in\F_2^{3|E(G)|}$ such that
\[
A^{\mathsf T}y=0
\qquad\text{and}\qquad
y^{\mathsf T}\mathbf d=1.
\tag{7.2}
\]
Consequently, $A\xi=\mathbf d$ is solvable if and only if
\[
y^{\mathsf T}\mathbf d=0
\qquad\text{for every }y\in\ker A^{\mathsf T}.
\tag{7.3}
\]
The tester notation below is the edge-block form of this criterion: the three
coordinates of $y$ belonging to edge $e$ are packaged as one vector
$r_e\in\F_2^3$. The rest of the proof gives an elementary description of
$\ker A^{\mathsf T}$ and verifies \textup{(7.3)} without invoking any further
duality theorem.

Every vector equation in \textup{(6.2)} is three scalar bit equations. A
convenient way to choose an XOR of some of those three equations is to use a
\emph{parity tester}.

\begin{definition}
For a tester $r=(r_1,r_2,r_3)\in\bits$ and a label
$s=(s_1,s_2,s_3)\in\bits$, define
\[
r\cdot s=r_1s_1\xor r_2s_2\xor r_3s_3\in\{0,1\}.
\tag{7.4}
\]
Here the products $r_is_i$ are ordinary products of bits, equivalently
products in $\F_2$. Applying $r$ to a three-bit equation means taking the XOR
of precisely those coordinate equations for which the corresponding bit of
$r$ is one. Thus the eight testers in $\bits$ are in bijection with the eight
subsets of the three coordinate equations; the tester $000$ selects none of
them.
\end{definition}

For instance, the tester $r=101$ selects the first and third coordinate
equations, and $101\cdot110=1\xor0\xor0=1$. Directly from \textup{(7.4)},
the dot product is linear in each argument:
\[
r\cdot(s\xor s')=(r\cdot s)\xor(r\cdot s'),
\qquad
(r\xor r')\cdot s=(r\cdot s)\xor(r'\cdot s).
\]
Also, if $r\cdot s=0$ for every $s\in\bits$, then $r=000$: taking $s$ to be
$100$, $010$, and $001$ reads off the three coordinates of $r$.

Assign one tester $r_e\in\bits$ to every edge equation \textup{(6.2)}, apply
$r_e$ to that equation, and XOR the resulting scalar equations over all
edges. The unknowns cancel precisely under the conditions in the next
lemma.

\begin{lemma}[Shape of an inconsistency certificate]
\label{lem:certificate-shape}
The triangle-gluing system is inconsistent if and only if there is a family
of testers $r_e\in\bits$, one for each edge, such that
\[
r_e\cdot f(e)=0
\qquad(e\in E(G)),
\tag{7.5}
\]
\[
\bigxor_{e\in\dd(v)}r_e=000
\qquad(v\in V(G)),
\tag{7.6}
\]
and
\[
\bigxor_{e\in E(G)} r_e\cdot d_e=1.
\tag{7.7}
\]
\end{lemma}

\begin{proof}
Expand every equation \textup{(6.2)} into three scalar equations. By
Lemma~\ref{lem:xor-certificate}, inconsistency gives an XOR of scalar edge
equations whose unknown coefficients vanish and whose right side is one.
For each edge $e$, record which of its three coordinate equations were used;
their indicator vector is $r_e\in\bits$. The preceding bijection shows that
this records the selected scalar equations uniquely. If none of the three
equations for $e$ was used, then $r_e=000$.

Applying $r_e$ to the vector equation \textup{(6.2)} gives the scalar equation
$(r_e\cdot t_u)\xor(r_e\cdot t_v) \xor\epsilon_e(r_e\cdot f(e))=r_e\cdot d_e$.
Indeed, the two linearity identities above handle every XOR term. For the
scalar-multiplication term,
$r_e\cdot\bigl(\epsilon_e f(e)\bigr)
=\epsilon_e\bigl(r_e\cdot f(e)\bigr)$, as follows by checking the two possible
values $\epsilon_e=0,1$.

The bit $\epsilon_e$ occurs only in the equation for edge $e$. After tester
$r_e$ is applied, its coefficient is $r_e\cdot f(e)$. Cancellation of every
$\epsilon_e$ is therefore exactly \textup{(7.5)}.

Because $G$ is loopless, the vector $t_v$ occurs once in the equation of every
edge incident with $v$ and nowhere else. After all tested equations are
XORed, put
\[
R_v=\bigxor_{e\in\dd(v)}r_e\in\bits.
\]
The total contribution involving the three coordinate unknowns of $t_v$ is
\[
\bigxor_{e\in\dd(v)}\bigl(r_e\cdot t_v\bigr)
=R_v\cdot t_v
=(R_v)_1(t_v)_1\xor(R_v)_2(t_v)_2\xor(R_v)_3(t_v)_3.
\]
Consequently, all three coordinate unknowns $(t_v)_i$ have coefficient zero
if and only if $(R_v)_i=0$ for $i=1,2,3$, equivalently $R_v=000$. Thus
cancellation of all vertex-translation unknowns is exactly \textup{(7.6)}.

After all unknowns cancel, the XOR of the right sides is
$\bigxor_{e\in E(G)} r_e\cdot d_e$. For the resulting equation to be $0=1$,
this XOR must be one, giving \textup{(7.7)}.

Conversely, suppose testers satisfy \textup{(7.5)}--\textup{(7.7)}. Apply
$r_e$ to the equation for each edge $e$ and XOR the resulting scalar
equations. Conditions \textup{(7.5)} and \textup{(7.6)} cancel every
$\epsilon_e$ and every coordinate of every $t_v$, while \textup{(7.7)} makes
the XOR of the right sides equal to $1$. The result is $0=1$, so
Lemma~\ref{lem:xor-certificate} proves that the gluing system is inconsistent.
\end{proof}

Thus compatibility will follow if we prove that testers satisfying
\textup{(7.5)} and \textup{(7.6)} always give right-side parity zero rather
than one.

\section{The local tester}
\label{sec:local-testers}

Continue with a tester family satisfying the two unknown-cancellation
conditions \textup{(7.5)} and \textup{(7.6)}; condition \textup{(7.7)} is not
assumed. Fix a vertex $v$. Because $G$ is cubic and loopless, its local
ordering $a,b,c$ consists of exactly the three distinct edge objects in
$\dd(v)$. As before, put
\[
x=f(a),\qquad y=f(b),\qquad z=f(c)=x\xor y.
\]
Write $r_a,r_b,r_c$ for the testers attached to these three edges.
Conditions \textup{(7.5)} and \textup{(7.6)} say, respectively,
\[
r_a\cdot x=0,
\qquad
r_b\cdot y=0,
\qquad
r_c\cdot z=0,
\tag{8.1}
\]
and
\[
r_a\xor r_b\xor r_c=000.
\tag{8.2}
\]

Define the bit
\[
\lambda_v=r_b\cdot x.
\tag{8.3}
\]
The local discrepancy contribution at $v$ is
\[
\begin{aligned}
\bigxor_{e\in\dd(v)}\bigl(r_e\cdot g_{v,e}\bigr)
&=(r_a\cdot000)\xor(r_b\cdot x)\xor(r_c\cdot000)\\
&=\lambda_v.
\end{aligned}
\tag{8.4}
\]
Here we used the local offsets $g_{v,a}=000$, $g_{v,b}=x$, and
$g_{v,c}=000$ from \textup{(5.4)}, together with $r\cdot000=0$ for every
tester $r$.

The next lemma gives a graph-theoretic interpretation of this bit.

\begin{lemma}[Local tester parity]
\label{lem:local-tester-parity}
Under conditions \textup{(8.1)} and \textup{(8.2)},
\[
\lambda_v
\equiv\#\{e\in\dd(v):r_e\neq000\}\pmod2.
\tag{8.5}
\]
Thus the local discrepancy contribution is the parity of the number of
nonzero testers at the vertex.
\end{lemma}

\begin{proof}
The colors $x$ and $y$ are distinct and nonzero by
Lemma~\ref{lem:fano-line}. Over $\F_2$, two nonzero vectors are linearly
dependent only when they are equal, so $x$ and $y$ are linearly independent.
Their span is $\{000,x,y,x\xor y\}$ and therefore has four elements. Since
$\F_2^3$ has eight elements, there is a vector $w\in\bits$ outside this span.
The vector $w$ is not a linear combination of $x$ and $y$, so $x,y,w$ are
linearly independent. Three linearly independent vectors in the
three-dimensional space $\F_2^3$ form a basis.

A tester is uniquely determined by its three answers on this basis. Indeed,
if two testers have the same answers, their XOR has dot product zero with
$x,y,w$. Linearity then gives dot product zero with every vector in
$\F_2^3$, so the nondegeneracy observation following \textup{(7.4)} says that
their XOR is $000$. In particular, a tester is $000$ if and only if its
answer triple on $x,y,w$ is $(0,0,0)$.

Define the bits
\[
A=r_a\cdot w,
\qquad
B=r_b\cdot w.
\]
We determine all answers of the three testers. XORing both sides of
\textup{(8.2)} with $r_c$ gives $r_c=r_a\xor r_b$. By linearity of the dot
product, evaluating this identity at $x$ and using $r_a\cdot x=0$ gives
\[
r_c\cdot x=(r_a\cdot x)\xor(r_b\cdot x)=\lambda_v.
\]
Evaluating at $y$ and using $r_b\cdot y=0$ gives
\[
r_c\cdot y=(r_a\cdot y)\xor(r_b\cdot y)=r_a\cdot y.
\]
Finally, $r_c\cdot z=0$ and $z=x\xor y$, so
\[
0=(r_c\cdot x)\xor(r_c\cdot y)
=\lambda_v\xor(r_a\cdot y).
\]
Hence $r_a\cdot y=\lambda_v$. Evaluating $r_c=r_a\xor r_b$ at $w$ also
gives
\[
r_c\cdot w=(r_a\cdot w)\xor(r_b\cdot w)=A\xor B.
\]
On the basis $x,y,w$, the testers therefore have answer triples
\[
\begin{array}{c|ccc}
&x&y&w\\
\hline
r_a&0&\lambda_v&A\\
r_b&\lambda_v&0&B\\
r_c&\lambda_v&\lambda_v&A\xor B
\end{array}
\tag{8.6}
\]

If $\lambda_v=1$, each row has a $1$ among its first two entries. Therefore
all three testers are nonzero. The number of nonzero testers is three, whose
parity is one.

If $\lambda_v=0$, the first two columns are zero. The testers are nonzero
exactly when their last entries $A,B,A\xor B$ are one, because a tester is
zero exactly when all three of its basis answers are zero. The complete list
of possible triples of last entries, obtained from the four possible pairs
$(A,B)\in\{0,1\}^2$, is
\[
(A,B,A\xor B)\in\{(0,0,0),(0,1,1),(1,0,1),(1,1,0)\}.
\]
The first triple has no ones, and each of the other three has two. Hence the
number of nonzero testers is even. The cases $\lambda_v=0$ and
$\lambda_v=1$ exhaust the two possible values of the bit $\lambda_v$; in both
cases the parity is $\lambda_v$, proving \textup{(8.5)}.
\end{proof}

\section{All vertex triangles can be glued}
\label{sec:triangle-compatibility}

\begin{theorem}[Triangle compatibility]
\label{thm:triangle-compatibility}
Let $G$ be a loopless cubic multigraph with a fixed nowhere-zero three-bit
flow $f$ and fixed local edge orderings. Define the offsets $g_{v,e}$ by
\textup{(5.4)} and the discrepancies $d_e$ by \textup{(6.1)}. Then there
exist vectors $t_v\in\bits$, one for each vertex, and bits
$\epsilon_e\in\{0,1\}$, one for each edge, such that
\[
t_u\xor t_v\xor\epsilon_e f(e)=d_e
\qquad\text{for every edge }e=uv.
\]
Equivalently, the triangle-gluing system has a solution, and the translated
endpoint triangles assign the same side set to every edge.
\end{theorem}

\begin{proof}
Suppose for contradiction that no such vectors and bits exist, so the
triangle-gluing system is inconsistent. By
Lemma~\ref{lem:certificate-shape}, choose a tester $r_e\in\bits$ for every
edge $e$ satisfying \textup{(7.5)}--\textup{(7.7)}. Put $S=\{e\in E(G):r_e\neq000\}$.
We evaluate the right-side parity in \textup{(7.7)}.

Let $\mathcal I(G)=\{(v,e):v\in V(G),\ e\in\dd(v)\}$
be the set of vertex-edge incidences. Because $G$ is loopless, every edge
$e$ has two distinct endpoints. For each edge, temporarily denote these
endpoints by $u_e$ and $v_e$ in an arbitrary order. Thus $e$ gives exactly
the two incidences $(u_e,e)$ and $(v_e,e)$. By the definition
$d_e=g_{u_e,e}\xor g_{v_e,e}$ and linearity of the dot product,
\begin{align*}
\bigxor_{e\in E(G)}\bigl(r_e\cdot d_e\bigr)
&=\bigxor_{e\in E(G)}
\left[
\bigl(r_e\cdot g_{u_e,e}\bigr)
\xor
\bigl(r_e\cdot g_{v_e,e}\bigr)
\right]\\
&=\bigxor_{(v,e)\in\mathcal I(G)}
\bigl(r_e\cdot g_{v,e}\bigr)\\
&=\bigxor_{v\in V(G)}
\bigxor_{e\in\dd(v)}\bigl(r_e\cdot g_{v,e}\bigr).
\tag{9.1}
\end{align*}
Interchanging $u_e$ and $v_e$ only interchanges the two terms inside the
brackets. The remaining equalities merely reindex and regroup a finite XOR,
so no orientation or ordering of the incidences is involved. Parallel edges
remain distinct because the edge object $e$ is part of every incidence.

The chosen tester family satisfies \textup{(7.5)} and \textup{(7.6)}, so the
calculation in Section~\ref{sec:local-testers} applies at every vertex, using
the fixed local ordering there. Equation \textup{(8.4)} identifies the inner
XOR in \textup{(9.1)} with $\lambda_v$, and
Lemma~\ref{lem:local-tester-parity} gives
\[
\begin{aligned}
\bigxor_{e\in E(G)}\bigl(r_e\cdot d_e\bigr)
&=\bigxor_{v\in V(G)}\lambda_v\\
&\equiv\sum_{v\in V(G)}\lambda_v\pmod2\\
&\equiv\sum_{v\in V(G)}|S\cap\dd(v)|\pmod2.
\end{aligned}
\tag{9.2}
\]
Here XOR is addition modulo two, and the last congruence sums \textup{(8.5)}
over all vertices. Each edge of $S$ is counted twice in \textup{(9.2)}, once
at each endpoint; parallel edges remain distinct.
Consequently,
\[
\sum_{v\in V(G)}|S\cap\dd(v)|=2|S|\equiv0\pmod2.
\]
The left side of \textup{(7.7)} is itself a bit, so congruence to zero modulo
two forces it to equal $0$. This contradicts the assertion in
\textup{(7.7)} that it equals $1$. No tester family satisfying
\textup{(7.5)}--\textup{(7.7)} exists. By the equivalence in
Lemma~\ref{lem:certificate-shape}, the gluing system is consistent and hence
has a solution. Proposition~\ref{prop:edge-gluing-equation} then gives the
equivalent side-set agreement asserted in the theorem.
\end{proof}

The proof uses neither connectedness nor bridgelessness once the nowhere-zero
three-bit flow has been fixed; it still uses cubicity in the local calculation
and looplessness in counting two distinct endpoint incidences per edge.
Lemma~\ref{lem:certificate-shape} says that every possible obstruction is a
tester certificate, and the incidence count above proves that the right side
of every such certificate is zero rather than one.

\section{From glued triangles to cycles}
\label{sec:cycles}

Continue with the loopless cubic multigraph $G$, the fixed nowhere-zero
three-bit flow and local edge orderings, and choose a solution
$(t_v,\epsilon_e)$ supplied by
Theorem~\ref{thm:triangle-compatibility}. For every incidence $(v,e)$, set
\[
Q_{v,e}
=Q_{v,e}(t_v)
=\{t_v\xor g_{v,e},\ t_v\xor g_{v,e}\xor f(e)\}.
\tag{10.1}
\]
By Lemma~\ref{lem:uniform-side-formula}, this is the triangle side assigned to
$e$ at $v$, and it has two distinct elements.
Proposition~\ref{prop:edge-gluing-equation} implies
\[
Q_{u,e}=Q_{v,e}
\qquad\text{for every edge }e=uv.
\]
Define the common two-element set by
\[
P_e=Q_{u,e}=Q_{v,e}.
\tag{10.2}
\]

\begin{lemma}[Labels trace cycles]
\label{lem:labels-trace-cycles}
For a fixed label $s\in\bits$, let $M_s=\{e\in E(G):s\in P_e\}$.
Let $G_s=(V(G),M_s)$ be the spanning subgraph of $G$ with this edge set.
Here each member of $M_s$ retains its identity as an edge object, so parallel
edges remain distinct. Every vertex has degree zero or two in $G_s$.
Consequently, every connected component of $G_s$ containing an edge is a
cycle.
\end{lemma}

\begin{proof}
Fix a vertex $v$. Its three incident edges are distinct edge objects, even
if some are parallel. For every $e\in\dd(v)$, equality \textup{(10.2)} gives
$P_e=Q_{v,e}$, so the three global pairs on the incident edges are exactly the
three translated triangle sides at $v$. By
Lemma~\ref{lem:triangle-parity}, the label $s$ belongs to exactly zero or two
of those side sets. By the definition of $M_s$, this is the number of edge
objects of $M_s$ incident with $v$. Since $G_s$ is loopless, that number is
precisely the degree of $v$ in $G_s$. Thus every vertex has degree zero or
two in $G_s$.

Let $K$ be a connected component of $G_s$ containing an edge. No vertex of
$K$ has degree zero in $G_s$, since such a vertex would be an isolated
component rather than a member of the edge-containing component $K$. Every
vertex of $K$ therefore has degree two in $G_s$. All $G_s$-edges incident
with a vertex of $K$ lie in the same connected component, so its degree in
$K$ is also two. Thus $K$ is a connected 2-regular subgraph, which is a
cycle by the definition in Section~1. This includes a two-edge cycle formed
by a pair of parallel edges.
\end{proof}

\pagebreak[3]
\begin{theorem}[Flow-to-CDC theorem]
\label{thm:flow-to-cdc}
Every loopless cubic multigraph that admits a nowhere-zero three-bit flow has
a cycle double cover. Moreover, the occurrences in the cover can be indexed
by pairs $(s,K)$, where $s\in\F_2^3$ and $K$ is an edge-containing component
of the label subgraph $G_s$.
\end{theorem}

\begin{proof}
Fix a nowhere-zero three-bit flow. At every vertex, choose an arbitrary
ordering of its three incident edge objects and build the corresponding
vertex triangle. Theorem~\ref{thm:triangle-compatibility} supplies a solution
to the gluing system, producing the global two-element edge sets $P_e$
defined in \textup{(10.2)}.

For each of the eight labels $s$, include every edge-containing connected
component $K$ of $G_s$ in a multiset $\mathcal C$. More precisely, an
occurrence of a cycle in $\mathcal C$ is indexed by the pair $(s,K)$. Thus if
the same cycle subgraph occurs for two different labels, it contributes two
distinct occurrences to the multiset. By
Lemma~\ref{lem:labels-trace-cycles}, every included component is a cycle.
The multiset is finite because there are eight labels and each finite graph
$G_s$ has finitely many connected components.

Fix an edge $e$. Its two-element set $P_e$ contains two distinct labels, say
$s$ and $s'$. From the definition of $M_q$, for $q\in\bits$, the edge $e$
belongs to $M_q$ if and only if $q\in P_e$. Hence $e$ belongs to exactly the
two spanning subgraphs $G_s$ and $G_{s'}$. In each it belongs to exactly one
connected component. Those components contain $e$, so both are included in
$\mathcal C$, giving one occurrence indexed by $s$ and one indexed by $s'$.
These remain two occurrences even if the two component subgraphs are equal.
For every other label $q\notin\{s,s'\}$, the edge $e$ is not in $M_q$, so no
other member of $\mathcal C$ contains $e$. Therefore every edge occurs in
exactly two members of $\mathcal C$, counted with multiplicity, and
$\mathcal C$ is a cycle double cover.
\end{proof}

\begin{corollary}[Cubic conclusion of the argument]
\label{thm:cubic-conclusion}
Every loopless bridgeless cubic multigraph has a cycle double cover.
\end{corollary}

\begin{proof}
Corollary~\ref{cor:three-bit-flow} supplies a nowhere-zero three-bit flow, so
Theorem~\ref{thm:flow-to-cdc} applies.
\end{proof}

The indexing in Theorem~\ref{thm:flow-to-cdc} uses eight possible labels, not
necessarily eight cycles. For one label $s$, the subgraph $G_s$ may have
several edge-containing components, and each such component contributes a
separate cycle occurrence indexed by $(s,K)$.

\section{Concluding remarks}
\label{sec:conclusion}

Our proof has four stages. First,
Proposition~\ref{prop:cubic-reduction} reduces the problem for finite
bridgeless multigraphs to the loopless bridgeless cubic case. Its construction
works componentwise, removes and later restores loops, expands every
nonisolated vertex into a replacement cycle, and projects a cycle double cover
of the resulting cubic multigraph back to the original graph.

Second, Corollary~\ref{cor:three-bit-flow}, obtained from the theorems of
Seymour and Tutte, supplies a nowhere-zero $\F_2^3$-flow on the cubic graph.
Third, Sections~\ref{sec:vertex-triangle}--\ref{sec:triangle-compatibility}
convert that flow into locally labeled vertex triangles and prove that the
triangles can be translated so that their side-label sets agree across every
edge. This is the content of
Theorem~\ref{thm:triangle-compatibility}. Fourth,
Section~\ref{sec:cycles} extracts cycles from the compatible labels and proves
the flow-to-cover result, Theorem~\ref{thm:flow-to-cdc}. Applying the
three-bit flow input gives the cubic conclusion,
Corollary~\ref{thm:cubic-conclusion}.

The main compatibility step is a finite binary-equation argument.
Proposition~\ref{prop:edge-gluing-equation} converts side agreement into the
triangle-gluing system. Lemma~\ref{lem:certificate-shape} shows that an
inconsistent system would have edge testers satisfying
\textup{(7.5)}--\textup{(7.7)}. Lemma~\ref{lem:local-tester-parity}
identifies each vertex contribution with the parity of its incident nonzero
testers. Theorem~\ref{thm:triangle-compatibility} then counts every nonzero
edge tester once at each endpoint, forcing the certificate's right-side
parity to be zero rather than one.

Once compatibility is established, each of the eight three-bit labels
selects a subgraph of degree zero or two at every vertex. Its edge-containing
components are cycles, and the two distinct labels on each edge place that
edge in exactly two indexed cycle components.

\begin{theorem}[Cycle Double Cover Theorem]
\label{thm:cdc}
Every finite bridgeless multigraph has a cycle double cover.
\end{theorem}

\begin{proof}
Corollary~\ref{thm:cubic-conclusion} establishes the hypothesis of
Proposition~\ref{prop:cubic-reduction}: every loopless bridgeless cubic
multigraph has a cycle double cover. Applying that proposition gives a cycle
double cover of every finite bridgeless multigraph.

For clarity, all exceptional cases are already included in the proposition.
Its proof works separately in disconnected components, ignores isolated
vertices, and handles a graph with no non-loop edges directly. In the
remaining case it constructs a loopless bridgeless cubic multigraph, applies
Corollary~\ref{thm:cubic-conclusion}, and discards cover cycles that use no old
edge. It projects every remaining cover cycle to a closed trail, decomposes
that trail into cycles, and finally restores every original loop by inserting
two copies of that one-edge cycle. Thus no connectedness, looplessness, or
minimum-degree assumption is being added to the theorem stated here.
\end{proof}

{\bf Acknowledgements}: 
A high level four-stage strategy was written by the author and provided to Codex and asked it to apply the strategy to open problems related to the Cycle Double Cover Conjecture (CDCC). Codex listed eight related open problems. The author used multiple rounds of prompting with some graph theoretic ideas and nudging Codex in the right direction by making it work out the steps on several examples of graphs, fill the proof gaps and fix the errors. The author then asked Codex to analyze the previously generated proofs and find missing steps. In later rounds of prompting, the author provided more graph theoretic details and iteratively asked Codex to verify all the steps section by section.

During this same week, OpenAI also announced a proof of CDCC \cite{OpenAI-CDC}. There are differences between our proof and OpenAI's proof. OpenAI's proof uses an abstract map and its annihilator in the dual space. We use scalar XOR equations, Gaussian elimination certificates, dot products and tester vectors. OpenAI's proof uses nowhere-zero 8-flow results of Kilpatrick and Jaeger and Tutte's group-flow equivalence. Our proof uses Seymour 6-flow as a starting point. Both these flows finally produce the same objects. Both papers share the high-level {\em flow to label} arguments. OpenAI's proof is linear-algebraic while our proof is graph-theoretic and combinatorial in nature.

We believe these techniques (triangles, edge gluing, inconsistency certificates, cycle extraction) might be useful to prove other related conjectures including Strong Cycle Double Cover Conjecture.

\end{document}